\documentclass[11pt]{amsart}

\usepackage[utf8x]{inputenc}

\usepackage{amssymb,latexsym,amsmath,amsxtra,amsthm,amsfonts,mathrsfs}
\usepackage{url}
\usepackage{amscd}
\usepackage{color}
\usepackage{epsfig}
\usepackage{slashed}
\usepackage{enumerate}

\usepackage{parskip}

\theoremstyle{plain}
\newtheorem{theorem}{Theorem}[section]
\newtheorem*{main}{Main Result}
\newtheorem{corollary}[theorem]{Corollary}
\newtheorem{lemma}[theorem]{Lemma}
\newtheorem{proposition}[theorem]{Proposition}

\theoremstyle{remark}
\newtheorem{remark}{Remark}
\theoremstyle{definition}

\newtheorem{Prob}{Problem}

\title[Pseudo-differential Operators and the Large Coupling Limit]{Pseudo-differential Operators, Transmission Problems and the Large Coupling Limit.}
\author{Ikemefuna C. Agbanusi}
\address{Department of Mathematics, University of Illinois at Urbana-Champaign}
\email{agbanusi@illinois.edu}
\thanks{}
\date{}

\newcommand{\lap}[1]{\Delta#1}
\newcommand{\bnorm}[1]{\left\lVert#1\right\rVert}
\newcommand{\ang}[1]{\langle#1\rangle}
\newcommand{\ind}{\mathbf{1}}
\newcommand{\at}[2]{\left.#1\right|_{#2}}
\newcommand{\paren}[1]{\left(#1\right)}
\newcommand{\norm}[1]{\Vert#1\Vert}
\newcommand{\dpair}[2]{\langle#1,#2\rangle}
\newcommand{\abs}[1]{\left|#1\right|}
\newcommand{\brac}[1]{\left[#1\right]}
\newcommand{\til}[1]{\widetilde{#1}}
\newcommand{\PD}[2]{\frac{\partial#1}{\partial#2}}

\def\N{\mathbb{N}} 
\def\R{\mathbb{R}} 
 
\def\C{\mathbb{C}}

\newcommand{\cS}{\mathcal{S}} 
\newcommand{\cF}{\mathcal{F}}

\newcommand{\sN}{\mathscr{N}_{\lambda}}
\newcommand{\sK}{\mathscr{K}}
\newcommand{\sKl}{\mathscr{K}_\lambda}
\newcommand{\sD}{\mathscr{D}_\lambda}
\newcommand{\sW}{\mathscr{W}_\lambda}

\DeclareMathOperator{\kers}{Ker}
\DeclareMathOperator{\dom}{Dom}
\DeclareMathOperator{\supt}{supp}
\DeclareMathOperator{\op}{Op}
\DeclareMathOperator{\ran}{Ran}

\numberwithin{equation}{section}
\begin{document}
\begin{abstract}
In this paper we prove some new results and give new proofs of known results related to the 
large coupling limit for stationary Schr\"odinger operators. The operators we consider are of the form 
$-\lap +\lambda V(x)$ where $\lap$ is the Laplacian, $V(x)$ is a real valued piecewise--constant potential 
having a jump discontinuity across a smooth interface and $\lambda$ is the coupling constant. Our main
result is that the potential determines a non-local boundary condition on the interface and we systematically
exploit this fact to derive various results about the \emph{large coupling problem}. In particular, we obtain 
estimates for convergence rates and a description of the behavior of the spectrum of $-\lap +\lambda V(x)$ as 
$\lambda\nearrow\infty$.
\end{abstract}

\maketitle

\section{Introduction and Statement of Results}
\subsection{Background}
Historically, the large coupling problem is to understand the behavior of operators
of the form $H_{\lambda}:=-\lap + \lambda V$ as $\lambda\nearrow\infty$. Here
$\lap :=\partial^2_{x_1}+\ldots+\partial^2_{x_n}$ is the Euclidean Laplacian; $V$ is the multiplication operator 
corresponding to a real-valued potential, $V(x)$; and $\lambda$ is a positive parameter called the 
\emph{coupling constant}. The term $\lap$ governs the ``free evolution" of a particle while $V$ describes its 
interactions\,---\,with an external field or other particles, for example. Informally, the coupling parameter modulates the strength
of the relevant interactions and as such the problem is really the description of quantum particles under very strong interactions.

Common questions are the existence and properties of the limit operator;
the rate and mode of convergence; the asymptotic behavior of the spectrum; and the description scattering phenomena, to name a few. A related problem
is the study of the various semigroups these operators generate. For instance, one
could consider the standard semigroups $e^{tH_\lambda}$, $e^{itH_\lambda}$, and $\cos t\sqrt{H_\lambda}$, corresponding to the heat,
Schr\"odinger and wave semigroups respectively, and attempt to describe them in the large coupling limit.

These problems could be classified under {semi-classical analysis} but perhaps ought to 
be construed as {singular perturbation} problems. We will not review the state-of-the-art of these problems but only
point out that pseudo-differential operators ($\Psi$DOs)\,---\,and the related micro-local analysis\,---\,have provided the main impetus
behind the progress in our current understanding. We refer the reader to the excellent survey by  {\sc Robert} \cite{Robert:1998vn} for an
account on this and related issues.

\subsection{Outline of Results and Methods}
Our goal in this work is to apply basic $\Psi$DO techniques to certain large coupling problems. In contrast to other work, we focus on Schr\"odinger operators 
defined on \emph{bounded domains} in $\R^n$ and we only consider special types of interaction potentials.  As will become clear, the restriction to bounded domains
allows us to employ rather specific tools.

More concretely, let $\Omega\subset\R^n$ be a bounded, open domain with smooth, 
i.e. $C^\infty$, boundary which we denote by $\Gamma$. Let $\Omega_1\Subset\Omega$ be a compact inclusion also 
with smooth boundary $\Gamma_1$. We define the ``exterior region" $\Omega_2 :=\Omega\setminus\overline{\Omega}_1$, so that 
$\Omega =\Omega_1\cup\Gamma_1\cup\Omega_2$ as in Figure \ref{fig:domain_fig}. Furthermore, we assume that $\Gamma_1$ and $\Gamma$ are locally on one side of 
$\Omega_1$ and $\Omega$ respectively.

\begin{figure}
\begin{center}
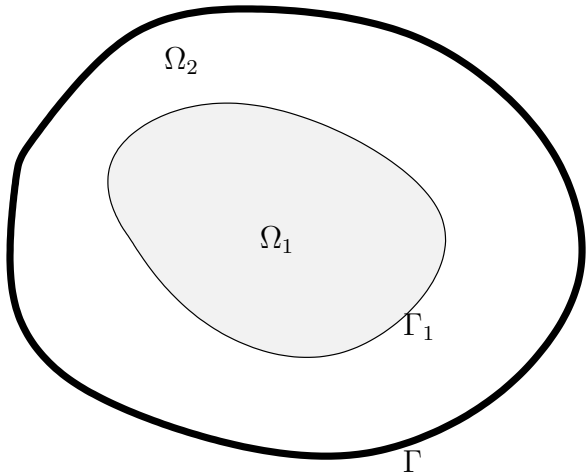
\end{center}
\caption{The domain $\Omega$}
\label{fig:domain_fig}
\end{figure}

Our  Schr\"odinger operator is of the form $A_{\lambda}:=-\lap+\lambda\ind_{{\Omega_1}}(x)$ with domain
\begin{equation}\label{eqn:dom_Doi}
\dom(A_{\lambda})=\{u(x)\in H^2(\Omega):\at{\partial_\nu u}{\Gamma}=0\}.
\end{equation}
Here $\ind_{E}(x)$ is the characteristic function of the set $E$, $\partial_\nu$ is the normal derivative and $H^k(\Omega)$
denotes the usual $L^2$ based Sobolev spaces. Note that the non-smooth interaction potential $\ind_{{\Omega_1}}(x)$ has
singularities along the interface $\Gamma_1$.  As for the usual large coupling problem on $R^n$, our first question is 
\begin{Prob}
Does $A_\lambda$ converge to a limit operator? If so, to what operator and at what rate?
\end{Prob}
For $\lambda>0$, standard results imply that the inverse $A_\lambda^{-1}$ exists and is bounded. Using quadratic forms for instance, one checks that $A_\lambda$ 
form a monotone sequence of operators. Abstract arguments (see {\sc Kato} \cite[Chap$.$ VIII, \S3]{Kato:1980kx}) then show the existence 
of a limit operator. A natural candidate for the limit operator is
 $A_{\infty}:=0\oplus B$ where
\begin{equation}\label{eqn:dom_Smol}
\dom(B)=\{v\in H^2(\Omega_2):\at{v}{\Gamma_1}=\at{\partial_\nu v}{\Gamma}=0\}; \quad Bv =-\lap v.
\end{equation}
We will later put this intuition on more solid ground; but these heuristics suggest \emph{large potentials} should be well approximated by Dirichlet boundary conditions.
This leads to the slightly more general problem of describing these operators for intermediate values of the coupling constant:
\begin{Prob}
Given $0<\lambda_0\leq\lambda<\infty$, find a boundary problem on the \emph{exterior domain}, $\Omega_2$, whose 
solutions closely approximate that of $A_\lambda u=f$.
\end{Prob}
One point of view is that boundary conditions are actually simplifications used to capture the fact that certain parameters in the system under study change rapidly
across an interface. Thus a solution to this question has practical implications for the numerical solution of PDEs where
boundary conditions are often difficult to implement. The same point has been made by {\sc Bardos--Rauch} in \cite{Bardos:1982fk} which
treats a large coupling problem for 1st order hyperbolic systems. This question also has implications for stochastic simulation algorithms
as shown in {\sc Agbanusi--Isaacson} \cite{Agbanusi:2014ys} and was one motivation for the study here. In fact, that paper deals with the time dependent version of our equations which were used in a model of diffusion to a stationary target.

What follows is one of our main results and gives an answer to the questions posed above. More precise statements will be made later:
\begin{main}\label{thm:main_result_boundary_pdo}
Let $0<\lambda_0\leq\lambda<\infty$ be fixed and suppose that  $f(x)$ has support in $\Omega_2$. If $u$ solves 
\begin{align*}
(-\lap+\lambda\ind_{_{\Omega_1}}(x))u&=f,\quad x\in\Omega;\\
\at{\partial_\nu u}{\Gamma}&=0,
\end{align*}
then $u_2$ defined by $u_2:=\at{u}{\Omega_2}$ satisfies the ``exterior" boundary value problem
\begin{equation}\label{eqn:ext_BVP}
\left\{\begin{aligned}
-\lap u_2 &=f,\quad x\in\Omega_2;\\
\at{u_2}{\Gamma_1}&=\mathscr{N}_{\lambda}\at{\paren{\partial_\nu u_2}}{\Gamma_1},\\
\at{\partial_\nu u_2}{\Gamma}&=0;
\end{aligned}
\right.
\end{equation}
where $\mathscr{N}_{\lambda}$ is a pseudodifferential operator depending on $\lambda$ and acting in $L^2(\Gamma_1)$.
\end{main}
There are good reasons for the particular assumptions on $f(x)$ in the above statement. For example, if one thinks of the corresponding diffusion equation, $f$ could be interpreted as a probability density/distribution. Thus the support condition on $f$ means that the distribution of the particles under consideration is outside the ``obstacle" $ \Omega_1$.

 As already hinted, in most applications, we wish to compare the operators $A_\lambda^{-1}$ and $B^{-1}$. A potential source of difficulty is their being defined on
different domains. To overcome this, we introduce the restriction operator:
\begin{equation*}
r_{_{\Omega_2}}f:=\at{f}{\Omega_2},
\end{equation*}
and the ``extension by zero" operator:
\begin{equation*}
e_{_{\Omega_2}}f:=
\begin{cases}
f, &x\in\Omega_2;\\
0, &x\in\Omega_1,
\end{cases}
\end{equation*}
and we observe that $r_{_{\Omega_2}}A^{-1}_{\lambda}e_{_{\Omega_2}}$ is now a bounded operator in $L^2(\Omega_2)$. Using our main result, we show
\begin{equation}
\norm{r_{_{\Omega_2}}A^{-1}_{\lambda}e_{_{\Omega_2}}-B^{-1}}_{op} =\mathcal{O}(\lambda^{-\frac{1}{2}}),
\end{equation}
as $\lambda\nearrow\infty$, which gives an estimate for the rate of convergence. The norm on the left is the operator norm and it is taken in $L^2(\Omega_2)$. 
A further consequence of our approach is that we obtain an estimate\,---\,cumbersome to state here (cf$.$ \eqref{eqn:spec_count_func_diff})\,---\,for
 $N(\mu;(r_{_{\Omega_2}}A^{-1}_{\lambda}e_{_{\Omega_2}}-B^{-1}))$. The function $N(\mu; T)$ is the \emph{spectral counting function} and it
counts the number of eigenvalues of the compact operator $T$ greater than $\mu$. These results show that 
one recovers the ``external" Dirichlet problem in the large coupling limit.

Our analysis rests on the observation that solutions to $A_\lambda u =f$ satisfy the following elliptic \emph{transmission problem}:
\begin{equation}\label{eqn:trans_PDE}
\left\{\begin{aligned}
(-\lap+\lambda)u_1 &=f_1;\quad x\in\Omega_1,\\
-\lap u_2 &=f_2;\quad x\in\Omega_2,
\end{aligned}
\right.
\end{equation}
where, for $i=1,2$, $f_i=\at{f}{\Omega_i}$; with the transmission condition on the interface $\Gamma_1$:
\begin{equation}\label{eqn:trans_BC}
\left\{\begin{aligned}
\at{u_1}{\Gamma_1}&=\at{u_2}{\Gamma_1},\\
\at{\partial_\nu u_1}{\Gamma_1}&=\at{\partial_\nu u_2}{\Gamma_1},
\end{aligned}
\right.
\end{equation}
and the ``external" boundary condition
\begin{align*}
\at{\partial_\nu u_2}{\Gamma}&=0.
\end{align*}
The proofs of our results are effected by constructing a parametrix for the system \eqref{eqn:trans_PDE}--\eqref{eqn:trans_BC}
 in a neighborhood of $\Gamma_1$. There are two key ideas here: the first, which goes
back {\sc Agmon} \cite{Agmon:1965fk}, is to treat $\lambda$ as an extra ``cotangent variable" in the parametrix construction; the second idea is to use 
a variant of the Calder\'on--Seeley--H\"ormander method of  reduction to the boundary (see, for instance, {\sc Chazarain--Piriou}
 \cite[Chap$.$ 5]{Chazarain:1982fk} for an exposition).
The ellipticity of the resulting equations and the transmission conditions \eqref{eqn:trans_BC} allow us to determine $\at{u}{\Gamma_1}$ and 
$\at{\partial_\nu u}{\Gamma_1}$ which in turn determine $\mathscr{N}_{\lambda}$ as a by-product. To apply this, we establish and exploit the following
Green's formula (cf$.$ Lemma \ref{lem:green_formula}) which may be of independent interest:
\[((r_{_{\Omega_2}}A^{-1}_{\lambda}e_{_{\Omega_2}}-B^{-1})f,g)_{L^2(\Omega_2)} =  -\dpair{\sN\gamma_{_1}u}{\gamma_{_1}v}_{L^2(\Gamma_1)};\quad f,g\in L^2(\Omega_2),\]
where $(\cdot,\cdot)_{L^2(\Omega_2)}$ and $\dpair{\cdot}{\cdot}_{L^2(\Gamma_1)}$ are inner products, $v=B^{-1}g$ and $u =r_{_{\Omega_2}}A^{-1}_{\lambda}e_{_{\Omega_2}}f$. 
The crucial thing is that our construction of the pseudo-differential operator $\mathscr{N}_{\lambda}$ comes with explicit
information on its symbol and it is an analysis of the dependence of $\mathscr{N}_{\lambda}$ on $\lambda$ which give the various estimates.

It is worthwhile to give another interpretation of our main result. Since the operators $\sN$ determine boundary
 conditions\,---\,albeit non-local ones\,---\,on $\Gamma_1$, we are entitled to view $\lambda$ as parametrizing a family of boundary value 
 problems in the exterior domain $\Omega_2$. Alternatively, 
these boundary problems correspond to certain \emph{realizations} of the Laplacian acting in $L^2(\Omega_2)$. Indeed our results show that the potential $\lambda\ind_{\Omega_1}(x)$ determines a one parameter family of relations, in this case graphs, in $H^{-\frac{1}{2}}(\Gamma_1)\times H^{\frac{1}{2}}(\Gamma_1)$. This could
be of independent interest and may allow for the application of other tools.
 We refer the reader to {\sc Grubb} \cite{Grubb:1968uq} for a thorough treatment of realizations of scalar elliptic differential
operators; and to {\sc Vishik} \cite{Vishik:1952kx} on which \cite{Grubb:1968uq} is based.

With some modifications, the ideas in this paper could be applied to large coupling problems for other $2^{nd}$ order equations and perhaps to similar large coupling problems on Riemannian manifolds\,---\,probably with more  substantial modifications in the latter case. The
method could  also be applied to the study of the resolvent \[R_\lambda(z) :=(A_\lambda - z)^{-1}=(-\lap+\lambda\ind_{_{\Omega_1}}(x)-z)^{-1},\] 
with $z\in \C$, as a prelude to studying the time dependent problems or developing a functional calculus in the large coupling limit.
Indeed one may view this paper as the study of $R_\lambda(0)$. We postpone these considerations to future papers.

\subsection{Other Work}
There are other approaches to the problem treated in this paper and we pause here to briefly review them. One standard
approach is to use asymptotic expansions and is due to {\sc Vishik--Lyusternik} \cite{Vishik:1960uo}. In that work, the authors consider 
much the same problem we do except they treat the situation where the exterior domain, which we have called $\Omega_2$, is unbounded. 
Using similar methods, {\sc Bruneau--Carbou} \cite{Bruneau:2002zr} have also obtained results on the asymptotic behavior of the eigenvalues
when the exterior domain is bounded.

There is a related approach using WKB expansions in dimension $1$ by {\sc Gesztesy et al} \cite[\S3]{Gesztesy:1988ua}. This paper treats more 
general potentials than we do and also obtain asymptotics for convergence rates as well as the spectral behavior in the 
large coupling limit. There is a rather robust formalism in {\sc BelHadjAli et al} \cite{BelHadjAli:2011ys}  capable of handling rather general Schr\"odinger type
operators perturbed by {measures}.

 The large coupling problem for the heat equation is treated in {\sc Demuth et al} \cite{M.Demuth:1995zr} using probabilistic methods and by the present author
 in \cite{Agbanusi:2013arx} using functional-analytic arguments. These papers treat the unbounded and bounded exterior domain cases respectively and we mention that in \cite{M.Demuth:1995zr} results on the stationary operators are derived from the time dependent operator by ``Laplace transform".

\section{Construction of $\sN$ and related operators}
\subsection{Notation and Preliminaries}
We first gather some notation to be used in addition to those employed in \S1.
Throughout, $\R^n$ is $n$--dimensional Euclidean space and a variable point will be written $x=(x',x_n)$ with $x'=(x_1,\ldots,x_{n-1})$. 
Elements of the dual space to $\R^n$ are written $\xi =(\xi',\xi_{n})$. We denote by $\R^n_+$ ($\R^n_-$) the half--space
defined by the relation $x_n>0$ ( $x_n<0$), and $\R^{n-1}$ the plane $x_n=0$.
\subsubsection{Sobolev Spaces}
As usual, $L^2(\R^n)$ is the space of equivalence class 
of measurable, square-integrable functions normed by
\[\norm{u}^2_{L^2(\R^n)} =\int\abs{u(x)}^2dx,\quad u \in L^2(\R^n).\]
We shall denote the Schwartz class of functions by $\cS(\R^n)$ or simply $\cS$, with its usual topology, and $\cS'$ its dual space of tempered
 distributions. Let $\mathcal{F}$ denote the Fourier transform $u\to\hat{u}$:
 \[\hat{u}(\xi)=\cF u(\xi) = \int e^{-ix\cdot\xi}u(x)\,dx,\quad u\in \cS.\]
 By the Plancherel theorem, $\cF$ can be extended to an isomorphism on $L^2$ and on $\cS'$ and Parseval's relation takes the form
 \[\norm{u}^2_{L^2(\R^n)} =(2\pi)^{-n}\norm{\hat{u}}^2_{L^2(\R^n)}.\]
 Let $\ang{\xi} = \paren{1+|\xi|^2}^{1/2}$ and  for $s\in\R$ we define the Sobolev spaces $H^s(\R^n)$ by
 \[H^s(\R^n):=\{u\in\cS':\ang{\xi}^s\hat{u}\in L^2(\R^n)\}.\]
 We denote the norm in $H^s(\R^n)$ by $\norm{\cdot}_s$ or $\norm{\cdot}_{s,\R^n}$ and identify $H^0$ with $L^2$. From 
 Plancherel's theorem, it follows that for $m\in\N$ we have
 \[H^m(\R^n):=\{u:D^\alpha u\in L^2(\R^n), \abs{\alpha}\leq m\},\]
 with the equivalent norm 
 \[\norm{u}_m^2 \simeq\sum_{\abs{\alpha}\leq m}\norm{D^\alpha u}^2.\]
 Here the derivatives are takes in the distribution sense and we employ standard multi--index notation:
 \[D^\alpha =(-i)^{|\alpha|}\partial^\alpha;\quad\partial^\alpha =\partial_{x_1}^{\alpha_1}\ldots\partial_{x_n}^{\alpha_n};\quad \alpha=(\alpha_1,\ldots,\alpha_n);\quad \abs{\alpha}=\sum_{j=1}^n\alpha_j.\]
 If $s=m+r$ with $m\in\N_0$ and $0<r<1$, one can show that the following defines an equivalent norm in $H^s(\R^n)$:
 \[\norm{u}_s^2\simeq \sum_{\abs{\alpha}\leq m}\norm{D^\alpha u}^2+ \sum_{\abs{\alpha}=m}\int\int|D^{\alpha} u(x) -D^{\alpha} u(y)|^2|x-y|^{-n-2r}\,dx\,dy.\]
 
 We can similarly define Sobolev spaces on $\R^n_\pm$. Briefly for $m\in\N$, $0<r<1$ and $s=m+r$
 \begin{align*}
 H^m(\R^n_{\pm})&=\{u\in L^2(\R^n_{\pm}): \sum_{\abs{\alpha}\leq m}\norm{D^\alpha u}^2_{L^2(\R^n_{\pm})}<\infty\},\\
 H^{s}(\R^n_{\pm})&=\{u\in H^m(\R^n_{\pm}):[D^{\alpha}u]_{r,\R^n_{\pm}}<\infty; \abs{\alpha}= m\},
 \end{align*}
 where we have defined the semi--norms
 \[[D^{\alpha}u]^2_{r,\R^n_{\pm}}=\int_{\R^n_{\pm}}\int_{\R^n_{\pm}}|D^{\alpha} u(x) -D^{\alpha} u(y)|^2|x-y|^{-n-2r}\,dx\,dy.\]
 We denote partial Fourier transforms by
 \[\tilde{u}(\xi',x_n)=\cF' u(\xi',x_n) = \int e^{-ix'\cdot\xi'}u(x)\,dx',\]
 and we consistently use primes to denote tangential variables or operators, i.e. variables or operators in $\R^{n-1}$. For example,
 $D'$ or $D_{x'}$ denotes differentiation in the $x'$ variables only, and so on.  We will need following theorem:
 \begin{theorem}[Trace Theorem]\label{thm:trace_thm}
For $u\in\cS$  and nonnegative integers $j$, the maps \[\gamma_{_j}u:=\at{(\partial^j_{x_n}u)}{x_n=0}\] satisfy
\[ \cF'\brac{\gamma_{_j}u}(\xi')=\int(i\xi_n)^j\hat{u}(\xi',\xi_n)\,d\xi_n.\]
Moreover, for $s>j+1/2$ and a constant depending only on $s$
\[\norm{\gamma_{_j}u}_{s-j-\frac{1}{2},\R^{n-1}}\leq C\norm{u}_{s},\]
and $\gamma_{_j}$ extends to a bounded surjection $\gamma_{_j}:H^s(\R^n)\to H^{s-j-\frac{1}{2}}(\R^{n-1})$.
 \end{theorem}

The fractional Sobolev spaces on the boundary are defined via local charts and a partition of unity from the
similar spaces defined in $R^{n-1}$. Using this we can extend the definition so that the trace operators 
$\gamma_{_j}v =\at{(\partial^j_{\nu}v)}{\Gamma_1}$ define bounded surjections from $H^s(\Omega_i)$ to 
$H^{s-j-\frac{1}{2}}(\Gamma_1)$ for $s>j+\frac{1}{2}$. Note that we use the same notation to define taking the traces from
``either side" of $ \Gamma_1$. We refer the reader to \cite[Chap$.$ 2]{Chazarain:1982fk} or {\sc Adams} \cite{Adams:1975fk} for more details.

\subsection{Symbols and Pseudo-differential operators}\label{subsec:symbl_calc}
For real $m$ and $k$ a non-negative integer, the function $a(x,\xi)$ belongs to the symbol class $S^m_k(\R^{2n})$ if $a\in C^{\infty}(\R^n\times\R^n)$ and
for any multi-indices $\alpha$ and $\beta$ there is are positive constants $C_{\alpha\beta}$ such that 
\[|\partial^\beta_x\partial^\alpha_{\xi}a(x,\xi)|\leq C_{\alpha\beta} (1+|\xi|)^{m-|\alpha|}; \quad \abs{\alpha}\leq k.\]
If $k\geq j$ then $S^m_k \subseteq S^m_j$ and the usual symbol class,  $S^m$, is characterized by $S^m = \bigcap_{k=0}^\infty S^m_k$. 
Important for us is a class of parameter dependent symbols which we now define. The symbol class $P^{m}_k(\R^n\times\R^n)$ consists
of  $b\in C^{\infty}(\R^n\times \R^n\times\R_+)$ for which there exists a $\lambda_0>0$ such that for $\lambda\geq\lambda_0$,
\[|\partial^\beta_x\partial^\alpha_{\xi}b(x,\xi,\lambda)|\leq C_{\alpha\beta} (|\xi|+\lambda^{\frac{1}{2}})^{m-|\alpha|};\quad \abs{\alpha}\leq k.\]
Since we aim to work locally we shall always assume that the symbols are either compactly supported in $x$ or else do not depend on $x$ outside some ball (which may depend on the symbol).

Associated to a symbol in either symbol class is a pseudo-differential operator defined by
\[a(x,D)u =\op(a)u(x) =(2\pi)^{-n}\int e^{ix\cdot\xi}a(x,\xi)\hat{u}(\xi)\,d\xi.\]
Basic facts about these operators, at least for symbols in $S^m$, can be found in \cite[Chap$.$ 4]{Chazarain:1982fk}. The composition rule is 
pertinent for our purposes and we recall that the rule relies on the observation that products of symbols of the same type are also symbols. That is,
if $a_i\in S^{m_i}$ and $b_i\in P^{r_i}$, for $i=1,2$, then $a_1a_2$ and $b_1b_2$ belong to $S^{m_1+m_2}$ and $P^{r_1+r_2}$ respectively. Later, we will need to compose
operators with symbols of \emph{different types}\,---\,one with a parameter and one without. We include some results in this direction as we have not 
found the exact statements we need in the existing literature. Nevertheless, some related results can be found in {\sc Agranovich} \cite{Agranovich:1971fk} and 
{\sc Grubb} \cite{Grubb:1978uq}.
 
Using Leibniz's rule and the inequality
\begin{equation}\label{eqn:param_and_non_param_symbl_ineq}
 (1+\abs{\xi})\leq (\abs{\tau}+\abs{\xi})\leq(1+\abs{\tau})(1+\abs{\xi});\quad \tau\in \C;\quad\abs{\tau}\geq 1,
\end{equation}
 we see that for $|\alpha|\leq m_1$ 
\begin{align*}
|\partial^\beta_x\partial^\alpha_{\xi}[a(x,\xi)b(x,\xi,\lambda)]|&\leq\sum_{\substack{\gamma\leq\alpha\\\mu\leq\beta}}C_{\gamma,\mu}|\partial^\mu_x\partial^\gamma_{\xi}a(x,\xi)\partial^{\beta-\mu}_x\partial^{\alpha-\gamma}_{\xi}b(x,\xi,\lambda)| \\
&\leq\sum_{\substack{\gamma\leq\alpha\\\mu\leq\beta}}\til{C}_{\gamma,\mu}(1+|\xi|)^{m_1-|\gamma|}(|\xi|+\lambda^{\frac{1}{2}})^{m_2+|\gamma|-|\alpha|} \\
&\leq C(|\xi|+\lambda^{\frac{1}{2}})^{m_1+m_2-|\alpha|}.
\end{align*}
Close examination of the computation above shows we have established
\begin{lemma}\label{lem:prod_symbl}
Suppose that $a\in S^{m_1}$ and $b\in  P^{m_2}$. If $m_1\geq0$ then $ab\in P^{m_1+m_2}_{[m_1]}$ and if $m_2\leq 0$ then $ab\in S^{m_1+m_2}$.
\end{lemma}

The next result describes the action of pseudo-differential operators with parameter dependent symbols. In the statements and proofs we replace $\lambda^{\frac{1}{2}}$ with $\tau$ to make 
the formulae less unwieldy.
\begin{proposition}\label{prop:map_neg_par_psdo}
Suppose that $b\in P_0^m$ with $m\leq0$. Let $r\in\R$ and $r+m\leq s\leq r$. Then $\op(b)$ can be extended to a bounded operator from $H^r(\R^n)$ to $H^{s-m}(\R^n)$ and we have
\begin{equation}
\bnorm{\op(b)u}_{s-m}\leq\frac{C}{\tau^{r-s}}\bnorm{u}_{r},
\end{equation}
for some constant independent of $u$ and $\tau$. In particular, \[\bnorm{\op(b)u}_{r}\leq\frac{C}{\tau^{|m|}}\bnorm{u}_{r}.\]
\end{proposition}

\begin{proof}
This is a variant of the proof of the boundedness of ``classical" pseudo-differential operators. Throughout we take $u\in\mathcal{S}$ to avoid any convergence issues and justify switching the order of integration. Since $b$ vanishes for $x$ outside some compact set,
it follows that
\begin{align*}
\abs{\sigma^\beta\int e^{-ix\cdot\sigma}b(x,\xi,\tau)\,dx}& =\abs{\int e^{-ix\cdot\sigma}D_x^\beta b\,dx} \leq C(|\xi|+\tau)^{m}\int\limits_{\supt(b)}dx
\end{align*}
which implies the Paley-Wiener type estimate for the Fourier transform of $b$: for any integer $N>0$
\[\hat{b}(\sigma,\xi,\tau)\leq C_N (|\xi|+\tau)^{m} (1+|\sigma|)^{-N}.\]
Writing $Bu$ for $\op(b)u$ we see that 
\[\widehat{Bu}(\xi) = \int \hat{b}(\xi-\sigma,\sigma,\tau)\hat{u}(\sigma)\,d\sigma,\]
and thus
\begin{align*}
\norm{Bu}^2_{s-m} =\int \ang{\xi}^{s-m}|\widehat{Bu}(\xi)|^2\,d\xi&=\int\abs{\int\ang{\xi}^{s-m} \hat{b}(\xi-\sigma,\sigma,\tau)\hat{u}(\sigma)\,d\sigma}^2\,d\xi\\
&:=\bnorm{\int G(\xi,\sigma)f(\sigma)\,d\sigma}_{L^2(\R^n_{\xi})},
\end{align*}
where $f(\sigma) =\ang{\sigma}^r\hat{u}(\sigma)$ and $G(\xi,\sigma)=\ang{\xi}^{s-m} \hat{b}(\xi-\sigma,\sigma,\tau)\ang{\sigma}^{-r}$. Using Peetre's inequality: 
$(1+|\xi|)^{t}(1+|\sigma|)^{-t}\leq (1+|\xi-\sigma|)^{|t|}$
we see that 
\begin{align*}
|G(\xi,\sigma)| &\leq C_N (|\sigma|+\tau)^{m} (1+|\xi-\sigma|)^{-N}(1+|\xi|)^{s-m}(1+|\sigma|)^{-r}\\
&\leq C_N(1+|\xi-\sigma|)^{-N+|s-m|} (|\sigma|+\tau)^{s-r}\paren{\frac{1+|\sigma|}{|\sigma|+\tau}}^{|m|+s-r}.
\end{align*}
Since $r+m\leq s\leq r$, choosing $N$ sufficiently large we see
\[\int |G(\xi,\sigma)|\,d\sigma \leq \frac{C}{\tau^{r-s}},\quad\text{and that}\quad \int |G(\xi,\sigma)|\,d\xi \leq \frac{C}{\tau^{r-s}}.\]
The result now follows by H\"older's inequality.
\end{proof}
The next result can be viewed as a consequence of the above proof or, more directly, of inequality \eqref{eqn:param_and_non_param_symbl_ineq}.
\begin{corollary}
Let $\Psi^{m}$ be the collection of pseudo-differential operators of order $m$. If $b\in P^m$ with $m\leq0$,
then $\op(b)\in\Psi^{m}$. In other words $\op(P^m)\subset\op(S^m)$ for $m\leq0$.
\end{corollary}

Another simple consequence is
\begin{corollary}\label{cor:crude_comp}
Suppose that $a\in S^{m_1}$ and $b\in  P^{m_2}$ with $m_2\leq0$. Then
\[\bnorm{\op(a)\circ\op(b)u}_{r-m_1}\leq\frac{C}{\tau^{|m_2|}}\bnorm{u}_{r}.\]
\end{corollary}
When $m_1$ is also negative in Corollary \ref{cor:crude_comp}, $\op(a)\circ\op(b)$ is smoothing and has a small norm for  large $\tau$.
The next result refines Corollary \ref{cor:crude_comp} in certain respects

\begin{proposition}\label{prop:crude_symbl_calc}
 Let $a\in S^{m_1}$ and $b\in  P^{m_2}$ with $m_1>0$ and $m_1+m_2\leq0$. Then $\op(a)\circ\op(b)$ is a pseudo-differential operator
with symbol in $P^{m_1+m_2}_{[m_2]}$. In particular for $r\in\R$ and with  $t=r+1-m_1+[m_1]$
\begin{equation}
 \norm{\op(a)\circ\op(b)-\op({\sum_{|\alpha|=0}^{[m_1]}\frac{1}{\alpha!}\partial^\alpha_\xi a(x,\sigma)D^\alpha_x b(x,\sigma,\tau))}}_{t}\leq\frac{c}{\tau^{|m_2|}}\norm{u}_r
\end{equation}
\end{proposition}

\begin{proof}
Putting $A=\op(a)$ and $B=\op(b)$, we see
 \begin{align*}
ABu(x) &= \int\int e^{i x\cdot\xi}a(x,\xi)\hat{b}(\xi-\sigma,\sigma,\tau)\hat{u}(\sigma)\,d\sigma\,d\xi\\
&= \int\int  e^{i x\cdot\sigma}\hat{u}(\sigma)e^{i x\cdot(\xi-\sigma)}a(x,\xi)\hat{b}(\xi-\sigma,\sigma,\tau)\,d\xi\,d\sigma\\
&:=\int  e^{i x\cdot\sigma}c(x,\sigma,\tau)\hat{u}(\sigma)\,d\sigma
 \end{align*}
where we have interchanged the order of integration as well as defined 
\begin{equation}\label{eqn:sybol_of_compos}
c(x,\sigma,\tau)=\int  e^{i x\cdot\omega} a(x,\omega+\sigma)\hat{b}(\omega,\sigma,\tau)\,d\omega.
\end{equation}
Direct estimation as in the proof of Lemma \ref{lem:prod_symbl} gives
\begin{align*}
|\partial^{\alpha}_{\sigma}\partial^{\beta}_{x}c|&\leq  \sum_{\substack{\gamma\leq\alpha\\ \mu\leq\beta}}C_{\gamma\mu}\int\abs{\partial^{\beta-\mu}_{x}e^{i x\cdot\omega} \partial^{\gamma}_{\sigma}\partial^{\mu}_{x}a(x,\omega+\sigma)\partial^{\alpha-\gamma}_{\sigma}\hat{b}(\omega,\sigma,\tau)}d\omega\\
&\leq\sum_{\gamma,\mu}C\int (1+|\omega+\sigma|)^{m_1-|\gamma|} (1+|\omega|)^{-N+|\beta-\mu|}(\tau+|\sigma|)^{m_2-|\alpha-\gamma|}d\omega.
\end{align*}
Using once again Peetre's inequality, the fact $|\alpha|\leq m_1$ and  taking $N$ sufficiently large we get
\[|\partial^{\alpha}_{\sigma}\partial^{\beta}_{x}c(x,\sigma,\tau)|\leq C(\tau+|\sigma|)^{m_1+m_2-|\alpha|},\]
proving the first part of the proposition. The second part follows by Taylor expanding $a(x,\omega+\sigma)$ in \eqref{eqn:sybol_of_compos} and estimating the remainder using Corollary \ref{cor:crude_comp} or as in the estimates for $c$ above. We leave the details to the reader.
\end{proof}
Combined, these results allow us to develop a symbol calculus to handle pseudo-differential operators with symbols which may or may not depend on a parameter.
We have probably provided more detail than is necessary here as the results are really consequences of inequality \eqref{eqn:param_and_non_param_symbl_ineq}
and the usual boundedness theorems.

We end this discussion with some examples. If $a\in S^1$ and $b\in P^{-1}$ then the composition $\op(a)\circ \op(b)$ makes sense as a pseudo-differential operator of order $0$ and
Proposition \ref{prop:crude_symbl_calc} and its corollaries show that we can view the principal symbol $ab$ as a symbol $ab\in P^{0}_{1}$ or $ab\in S^{0}$.
If on the other hand  $a\in S^1$ and $b\in P^{-2}$, it is more helpful to think of $ab$ as belonging to $P_1^{-1}$ than to $S^{-1}$ since the former viewpoint implies special operator bounds.

\subsection{Determination of the Operators}
To lighten the exposition, we first demonstrate the existence of the operators on functional-analytic considerations. Later we will characterize them as pseudo-differential operators. We begin with the observation, used implicitly in the Introduction, that the operator $B$ is a symmetric operator with compact inverse. 
This is a consequence of the following well
known fact:
 \begin{lemma}[Poincar\'e Inequality]
Let $v\in H^{1}(\Omega_2)$ satisfy $\gamma_{_0}v=0$.
Then, for some constant $C>0$, 
\begin{equation*}
\norm{v}_{L^2(\Omega_2)}\leq C \norm{\nabla v}_{L^2(\Omega_2)}.
\end{equation*} 
\end{lemma}
The invertibility of $B$ allows us to define the Poisson operator $\sK$ which satisfies
\begin{equation}\label{eqn:Dir_poiss_op}
\left\{\begin{aligned}
-\lap(\sK\varphi) &= 0,\quad\text{in\,}\,\Omega_2;\\
\gamma_{_0}(\sK\varphi) &=\varphi,\quad\text{on\,}\,\Gamma_1;\\
\at{(\sK\varphi)}{\Gamma} &=0,
\end{aligned}
\right.
\end{equation}
for $\varphi\in H^{\frac{3}{2}}(\Gamma_1)$.

For $\lambda\geq1$, we define the operator $A^{-1}_{\lambda,\nu}$ which is the inverse of (the closure of) 
$-\lap+\lambda$ acting in $L^2(\Omega_1)$, i.e. in the ``interior domain",  with \emph{Neumann} boundary conditions on $\Gamma_1$. 
That is, for $f\in L^2(\Omega_1)$,the function $w=A^{-1}_{\lambda,\nu}f$ satisfies:
\begin{equation*}
\left\{\begin{aligned}
(-\lap+\lambda)w &= f,\quad\text{in\,}\,\Omega_1;\\
\gamma_{_1}w&=0,\quad\text{on\,}\,\Gamma_1.
\end{aligned}
\right.
\end{equation*}
Another standard functional analysis argument shows the existence of $A^{-1}_{\lambda,\nu}$. With this in mind, we may define the associated 
Poisson operator $\sKl$ which now solves, for $\varphi\in H^{\frac{1}{2}}(\Gamma_1)$,
\begin{equation}\label{eqn:Int_Neu_poiss_op}
\left\{\begin{aligned}
(-\lap+\lambda)(\sKl\varphi) &= 0,\quad\text{in\,}\,\Omega_1;\\
\gamma_{_1}(\sKl\varphi)&=\varphi,\quad\text{on\,}\,\Gamma_1.
\end{aligned}
\right.
\end{equation}
As discussed in the Introduction, we solve the equation $A_\lambda u=f$ by solving the transmission problem \eqref{eqn:trans_PDE}--\eqref{eqn:trans_BC}
which we recall for the readers convenience:
\begin{align*}
(-\lap+\lambda)u_1 &=f_1;\quad x\in\Omega_1,\\
-\lap u_2 &=f_2;\quad x\in\Omega_2\\
\gamma_{_0}u_1&=\gamma_{_0}u_2\\
\gamma_{_1} u_1&=\gamma_{_1} u_2\\
\at{\partial_\nu u_2}{\Gamma}&=0.
\end{align*}
We put $\varphi_0=\gamma_{_0}u_1$ and $\varphi_1=\gamma_{_1}u_1$ and we treat  $\varphi_0$ and  $\varphi_1$ as unknown functions.
It is easy to verify that
\begin{equation}\label{eqn:ext_soln}
u_2=B^{-1}f_2+\sK\varphi_0,
\end{equation}
 and
\begin{equation}\label{eqn:int_soln}
u_1=A^{-1}_{\lambda,\nu}f_1+\sKl\varphi_1
\end{equation}
furnish a solution to the transmission problem. We only have to determine the unknown boundary values  $\varphi_0$ and  $\varphi_1$. 
To this end, we apply the trace operators $\gamma_{_1}$ to \eqref{eqn:ext_soln} and  $\gamma_{_0}$ to \eqref{eqn:int_soln} to obtain 
the system of equations on $\Gamma_1$:
\begin{equation}\label{eqn:boundary_sys}
\begin{pmatrix}
Id &&-\gamma_{_0}\sKl \\
- \gamma_{_1}\sK&& Id
\end{pmatrix}
\begin{pmatrix}
\varphi_0\\
\varphi_1
\end{pmatrix}
=
\begin{pmatrix}
\gamma_{_0}A^{-1}_{\lambda,\nu}f_1\\
\gamma_{_1}B^{-1}f_2
\end{pmatrix}.
\end{equation}

Here $Id$ is the identity operator and $\gamma_{_0}\sKl$ is the composition of the two operators. Note that $\gamma_{_0}\sKl$  and $\gamma_{_1}\sK$ 
are well defined operators with  $\gamma_{_0}\sKl:H^{\frac{1}{2}}(\Gamma_1) \to H^{\frac{3}{2}}(\Gamma_1)$ and 
 $\gamma_{_1}\sK:H^{\frac{3}{2}}(\Gamma_1) \to H^{\frac{1}{2}}(\Gamma_1)$. The next theorem pushes the whole program through:
\begin{theorem}\label{thm:exist_of_PsDOs}
Let $\sN :=\gamma_{_0}\sKl$ and $\sD:= Id-\gamma_{_1}\sK\sN$. Then $\sN$ and $\sD$ are elliptic pseudodifferential operators of orders $-1$ and $0$ 
respectively acting in $H^{\frac{1}{2}}(\Gamma_1)$ . If we define $\sW =-\sN\sD^{-1}$, then $\sW$ is also an elliptic pseudodifferential operator of order $-1$.
In particular $\sN\in\op(P^{-1})$, $\sD\in\op(P^{0})$ and $\sW\in\op(P^{-1}_1)$.
\end{theorem}

We will later sketch the proof of this theorem in the remaining subsections. For now we show
\begin{corollary}\label{cor:main_result_corr}
 Suppose that $u$ solves $A_\lambda u = f$ and that $f_1=0$ i.e.,  $f$ has support in $\Omega_2$. Then it holds that
\begin{equation}\label{eqn:bound_soln}
\varphi_0=\sN\varphi_1,\quad\text{and}\quad \varphi_1=\sD^{-1}(\gamma_{_1}B^{-1}f_2), 
\end{equation}
or $\gamma_{_0}u_2=\sN\gamma_{_1}u_2$ where $u_i = \at{u}{\Omega_i}$ for $i=1,2$.
\end{corollary}
This is one of our main results advertised in the Introduction.
\begin{proof}
It follows that $A^{-1}_{\lambda,\nu}f_1=0$ and by  standard elliptic regularity theory,  $A^{-1}_{\lambda,\nu}f_1$ belongs to $H^2$.
By the trace theorem we have that $\gamma_{_0}A^{-1}_{\lambda,\nu}f_1=0$ and the first equation in \eqref{eqn:boundary_sys} shows that
$\varphi_0=\sN\varphi_1$. Now the second equation in  \eqref{eqn:boundary_sys} and the ellipticity of $\sD$ show that 
$\varphi_1=\sD^{-1}(\gamma_{_1}B^{-1}f_2)$.
\end{proof} 
The rest of this section is devoted to a sketch of the construction of $\sN$ as a pseudo-differential operator. We will skip some technical details to keep the
 paper to a reasonable length.

\subsection{Localization}
The statement of our main result is local and allows us to reduce, via a partition of unity and a local coordinate change, to considering a problem in the
 neighborhood of the origin in $\R^n$. 

More precisely, let $x_0$ be a point  in $\Gamma_1$. By assumption there is a local chart $U$ of $x_0$ and a $C^\infty$ change of coordinates
which locally ``flattens" the boundary. We write the change of variables as $y=\Phi(x)$ and we assume it is of the form, possibly after a rotation,
translation and relabeling:
\begin{equation}\label{eq:ChangeofVar}
\left\{\begin{aligned}
y_i &=x_i,&&1\leq i\leq n-1\\
y_i &= x_i -\chi(x'),&& i=n.
\end{aligned}
\right.
\end{equation}

The boundary $\Gamma_1$ is now identified with the plane $y_n=0$ after the coordinate change as in Figure \ref{fig:flatten_fig}.
By the well known change of coordinates formula we can rewrite the
Laplacian in such local coordinates as
 \begin{equation*}
P(y',D):= \sum_{j,k}{A_{jk}(y') \frac{\partial^2}{\partial_{y_{j}}\partial_{y_{k}}}} + \sum_{j}{b_j(y') \frac{\partial}{\partial_{y_j}}}
 \end{equation*}
 where
 \begin{equation*}\label{eq:LocCoeff}
A =
\begin{pmatrix}
I&&-\nabla_{x'}\chi\\
-(\nabla_{x'}\chi)^t&&1+|\nabla_{x'}\chi|^2
\end{pmatrix},\quad\text{and}\quad
b_j = 
\begin{cases}
0,&\text{for\,}\,j\leq n-1;\\
-\lap\chi, &\text{for\,}\,j= n.
\end{cases}
\end{equation*}
We note that $\det(A)=1$ and that $P(y',D)$ is uniformly elliptic.
The following lemma is nearly obvious and allows us to simplify the expression for the normal derivative. The proof is a consequence
of $\Gamma_1$ being non-characteristic which in turn is a consequence of ellipticity:
 \begin{lemma}\label{lem:normal_deriv}
 Let $g_1(x)$ and $g_2(x)$ be $C^1$ functions such that $\at{g_1}{\Gamma_1} = \at{g_2}{\Gamma_1}$, then 
 $\at{\partial_{n}g_1}{\Gamma_1}=\at{\partial_{n}g_2}{\Gamma_1}$ if and only if $ \partial_{x_n}g_1(x',\chi(x'))=\partial_{x_n}g_2(x',\chi(x'))$.
\end{lemma}
After relabeling our coordinates, we see that we must consider the following P.D.E in \emph{local coordinates}:
\begin{align*}
  (P(x',D)-\lambda)u_1&=f_1; \quad x\in \R^n_{+}\\
P(x',D) u_2&=f_2;\quad x\in \R^n_{-}\\
u_1(x',0) &= u_2(x',0)\\
\partial_{x_n}u_1(x',0)&=\partial_{x_n}u_2(x',0).
\end{align*}
We are only really interested in compactly supported solutions to the above equations since they arise out of our localization procedure. 
Hence we may assume that all the
data are supported in a ball $B_\delta(0)$ of radius $\delta$ near the origin.

\begin{figure}
\begin{center}
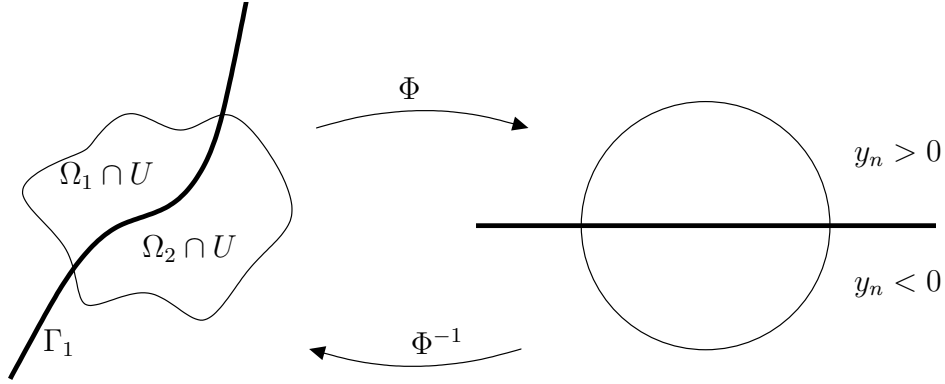
\end{center}
\caption{Flattening the boundary $\Gamma_1$ }
\label{fig:flatten_fig}
\end{figure}

The point is to reduce this to a study of ODEs in the normal variable $x_n$ by taking partial Fourier transform in $x'$. Thus we must consider 
the polynomial in $z$ with complex coefficients:
\begin{equation*}
p(x',\xi',z):=A_{nn}(x')z^2+2iz\sum_{k=1}^{n-1}A_{nk}(x')\xi_k-|\xi'|^2,
\end{equation*}
obtained by taking Fourier--Laplace transforms of the principal term of $P(x',D)$ in the tangential and normal directions respectively with ``frozen coefficients". 
Similarly let $q(x',\xi',\omega,\lambda):= p(x',\xi',\omega)-\lambda$, that is,
\[q(x',\xi',\omega,\lambda) =A_{nn}(x')\omega^2+2i\omega\sum_{k=1}^{n-1}A_{nk}(x')\xi_k-(|\xi'|^2+\lambda) \]
which the principal symbol of $P(x',D)-\lambda$ obtained by treating $\lambda$ as an extra cotangent variable. 
As we have mentioned, this idea goes back at least to \cite{Agmon:1965fk} and was further refined in {\sc Agranovich--Vishik} \cite{Agranovich:1964zr}, and in {\sc Seeley} \cite{Seeley:1968vn} using $\Psi$DO techniques. The roots of these polynomials are given by
\begin{equation}
z_{\pm} =\frac{1}{A_{nn}(x')}\paren{\pm\sqrt{A_{nn}(x')|\xi'|^2-\paren{\sum_{k=1}^{n-1}A_{nk}(x')\xi_k}^2}-i\sum_{k=1}^{n-1}A_{nk}(x')\xi_k},
\end{equation}
and \begin{equation}
\omega_{\pm} =\frac{1}{A_{nn}}\paren{\pm\sqrt{A_{nn}(x')(|\xi'|^2+\lambda)-\paren{\sum_{k=1}^{n-1}A_{nk}(x')\xi_k}^2}-i\sum_{k=1}^{n-1}A_{nk}(x')\xi_k}
\end{equation}
It is easy to check that $z_\pm(x',\xi')$ and $\omega_\pm(x',\xi',\lambda)$ are homogenous of degree $1$ in $\xi'$ and $(\xi,\lambda^{\frac{1}{2}})$
respectively and never vanish for $|\xi'|\neq0$. In particular, $\Re z_-<0<\Re z_+$ and $\Re\omega_-<0<\Re\omega_+$  for  $|\xi'|\neq0$.
For fixed $x\in B_\delta(0)$, one can show that
\[|z_\pm(x',\xi')|\leq C(1+|\xi'|^2)^{\frac{1}{2}},\]
 if $|\xi'|\geq1$ and that
\[|\omega_\pm(x',\xi',\lambda)|\leq C(\lambda+|\xi'|^2)^{\frac{1}{2}},\]
for $\lambda\geq1$, where the constants depend only on $\delta$ and the local coordinate map $\chi$.  Direct estimation of the derivatives, or
 availing ourselves of the homogeneity of the symbols, show that $z_\pm(x',\xi')$ and $\omega_\pm(x',\xi',\lambda)$ belong to the symbol classes 
 $S^1$ and $P^1$ respectively.

\subsection{The Poisson Operators}
Let $\psi_0(x,\xi')$ be smooth compactly supported function vanishing in a neighborhood of $\xi'=0$, identically 1 for $|x|\leq\delta$ 
and vanishing outside $|x|\geq 2\delta$. Let $\psi_1(x)$ be identically 1 for $|x|\leq\delta$ and vanishing outside $|x|\geq 2\delta$. 
Put $\tau(x',\xi') = z_+(x',\xi') $ and $\eta(x',\xi',\lambda)  =\omega_-(x',\xi',\lambda)$.

As a first step, we define, for $ x_n<0$, and $ x_n>0$ respectively, the operators
\begin{equation}
(K\varphi)(x):= \frac{1}{(2\pi)^{n-1}}\int\limits_{\R^{n-1}}e^{ix'\cdot\xi'}\psi_0(x,\xi')e^{x_n\tau(x',\xi')}\tilde{\varphi}(\xi')\,d\xi,
\end{equation}
and
\begin{equation}
(K_{\lambda}\varphi)(x):= \frac{1}{(2\pi)^{n-1}}\int\limits_{\R^{n-1}}e^{ix'\cdot\xi'}\psi_1(x)\frac{e^{x_n\eta(x',\xi',\lambda)}}{\eta(x',\xi',\lambda)}\tilde{\varphi}(\xi')\\,d\xi',
\end{equation}
where $\varphi(x')$ is supported in $|x'|\leq\delta$. These formulae arise out of the solution to the associated O.D.E. For instance if we put $h =K\varphi$ then $h$ satisfies the following O.D.E with frozen coefficients:
\begin{align*}
p(x',\xi',\partial_{x_n})\tilde{h}(\xi',x_n)&=0;\quad (x_n<0)\\
\tilde{h}(\xi',0) &=\tilde{\varphi}(\xi').
\end{align*}
It is not too difficult to show that the integrals defining the operators above are absolutely 
convergent if $\varphi(x')$ is nice, say $\varphi\in \cS(\R^{n-1})$. 

It is possible to show that 
$K:H^{\frac{3}{2}}(\R^{n-1}) \to H^{2}(\R^n_-)$ and $K_\lambda:H^{\frac{1}{2}}(\R^{n-1}) \to H^{2}(\R^n_+)$ are continuous maps and, as such, the above expressions admit
traces. That is:
\begin{align*}
\gamma_{_0}K\varphi:=(K\varphi)(x',0)&= \frac{1}{(2\pi)^{n-1}}\int\limits_{\R^{n-1}}e^{ix'\cdot\xi'}\psi_0(x',0,\xi')\tilde{\varphi}(\xi')\,d\xi'\\
&=\varphi +\op'(\psi_0(x',0,\xi')-1)\varphi,
\end{align*}
is well defined as is
\begin{align*}
\gamma_{_1}K_{\lambda}\varphi&= \frac{1}{(2\pi)^{n-1}}\int\limits_{\R^{n-1}}e^{ix'\cdot\xi'}\paren{\psi_1(x',0)+\frac{\partial_{x_n}\psi_1(x',0)}{\eta}}\tilde{\varphi}(\xi')\,d\xi'\\
&=\varphi + \op'\paren{\psi_1(x',0)-1+\eta^{-1}\partial_{x_n}\psi_1(x',0)}\varphi.
\end{align*}
Since $\varphi$ is supported in $B_\delta(0)$ and the cut-off functions $\psi_0$ and $\psi_1$ are identically $1$ there, the calculus of 
$\Psi$DO shows that the ``error" terms in the traces above are smooth for any $\varphi$.

If the compositions $P(x',D)\circ K=0$ and $(P(x',D)-\lambda)\circ K_\lambda=0$, then we are in business. However,
another computation shows that $P(x',D)\circ K=Q$ and also that $(P(x',D)-\lambda)\circ K_\lambda= Q_\lambda$ for some operators $Q$ and $Q_\lambda$.
It is possible to show that these operators, which are like error terms, map bounded sets
to compact sets. More precisely, $Q_\lambda:H^{\frac{1}{2}}(\R^{n-1}) \to H^{1}(\R^n_+)$ and $Q:H^{\frac{3}{2}}(\R^{n-1}) \to H^{1}(\R^n_-)$ are bounded operators.
Thus, modulo compact operators, $P(x',D)\circ K=0$ and $(P(x',D)-\lambda)\circ K_\lambda=0$. With some more work, one can construct local approximations, i.e. 
in a neighborhood of $\Gamma_1$, of $\sK$ and $\sKl$ with ``leading terms" $K$ and $K_\lambda$ respectively.

\subsection{The Symbols}
We can now directly verify that
\begin{align*}
\gamma_{_1}K\varphi&=\frac{1}{(2\pi)^{n-1}}\int\limits_{\R^{n-1}}e^{ix'\cdot\xi'}\tau(x',\xi')(\psi_0(x',0,\xi')+\partial_{x_n}\psi_0(x',0,\xi'))\tilde{\varphi}(\xi')\,d\xi'\\
&=\op'(\tau)\varphi + \op'(\psi_0(x',0,\xi')+\partial_{x_n}\psi_0(x',0,\xi')-1)\varphi
\end{align*}
and
\begin{align*}
\gamma_{_0}K_\lambda\varphi&=\frac{1}{(2\pi)^{n-1}}\int\limits_{\R^{n-1}}e^{ix'\cdot\xi'}\frac{\psi_1(x',0)}{\eta(x',\xi',\lambda)}\tilde{\varphi}(\xi')\,d\xi'\\
&=\op'(1/\eta)\varphi + \op'(\psi_1(x',0)-1)\varphi.
\end{align*}
Hence the principal symbol of $\sN$ is $1/\eta(x',\xi',\lambda)$ which never vanishes for $\lambda\geq1$. This shows that $\sN$ is elliptic 
and of order $-1$. Also  $\tau(x',\xi')$, the principal symbol of $\gamma_1\sK$, never vanishes for $|\xi'|\geq1$ and thus is also elliptic.
From the symbol calculus in \S\,\ref{subsec:symbl_calc}, it follows that the principal symbol of $ Id-\gamma_{_1}\sK\sN$ is given by 
\[1-\frac{\tau(x',\xi')}{\eta(x',\xi',\lambda)} = \frac{\Re\eta(x',\xi',\lambda)-\Re\tau(x',\xi')}{\eta(x',\xi',\lambda)},\]
which is seen to be elliptic and in fact uniformly bounded away from $0$ in modulus for $\lambda\geq1$. Direct estimation also shows that it 
belongs to the symbol class $P^0_1$
being homogenous of degree $0$ in $(\xi',\lambda^{\frac{1}{2}})$. Hence $\sD$ is seen to be invertible and of order $0$.

Since $\eta(x,\xi',\lambda)$ belongs to the symbol space $P^{-1}$ it follows from Proposition \ref{prop:map_neg_par_psdo} that
\begin{equation}
\norm{\mathscr{N}_{\lambda}\varphi}_{H^s(\R^{n-1})} \leq
\begin{cases}
\dfrac{C}{\sqrt{\lambda}}\norm{\varphi}_{ H^{\frac{1}{2}}(\R^{n-1})},&\text{for}\quad0\leq s\leq\frac{1}{2}\\
\dfrac{C}{\lambda^{\frac{3}{4}-\frac{s}{2}}}\norm{\varphi}_{ H^{\frac{1}{2}}(\R^{n-1})},&\text{for}\quad\frac{1}{2}\leq s\leq\frac{3}{2},
\end{cases}
\end{equation}
for $\lambda$ sufficiently large.

Finally, we note that the principal symbol of $\sW$ is given by 
\begin{equation}\label{eqn:weyl_symbl_intro}
(\Re\tau(x',\xi') -\Re\eta(x',\xi',\lambda))^{-1}
\end{equation}
 which shows that $\sW$ is semi-bounded and elliptic with real principal symbol and thus, modulo a regularizing operator, is self-adjoint.

\section{Applications}
\subsection{A Convergence Rate}
Recall the definition of the restriction operator and the extension operator:
 \begin{equation*}
r_{_{\Omega_2}}f=\at{f}{\Omega_2};\quad e_{_{\Omega_2}}f=
\begin{cases}
f, &x\in\Omega_2;\\
0, &x\in\Omega_1.
\end{cases}
\end{equation*}
It is a standard fact that $e_{_{\Omega_2}}:H^{s}(\Omega_2)\to H^{s}(\Omega)$ is bounded for $0\leq s\leq \frac{1}{2}$. Since differentiation is local, it is not too hard to show that $D^{\alpha}(r_{_{\Omega_2}}f) =r_{_{\Omega_2}}(D^{\alpha}f)$ for $|\alpha|=k\in\N$. Thus $r_{_{\Omega_2}}:H^{k}(\Omega)\to H^{k}(\Omega)$ is bounded and by interpolation, $r_{_{\Omega_2}}$ extends to a bounded operator in $H^{s}(\Omega)$ for $s\geq0$. Moreover the operators $e_{_{\Omega_2}}$ and $r_{_{\Omega_2}}$ are each others adjoints since they are both bounded and for $f\in L^2(\Omega)$ and $g\in L^2(\Omega_2)$
\[(r_{_{\Omega_2}}f,g)_{_{L^2(\Omega_2)}} =\quad\int_{\Omega_2}fg \quad=\quad\int_{\Omega}fe_{_{\Omega_2}}g \quad=\quad(f,e_{_{\Omega_2}}g)_{_{L^2(\Omega)}}.\] 
Here and in what follows we sometimes drop the differential, $dx$, in integrals. We can now state and prove the following theorem 

\begin{theorem}[The Large Coupling Limit]\label{thm:oper_conv_diff}
The operator $r_{_{\Omega_2}}A^{-1}_{\lambda}e_{_{\Omega_2}}-B^{-1}$ is compact and it holds that
\begin{equation*}
\norm{r_{_{\Omega_2}}A^{-1}_{\lambda}e_{_{\Omega_2}}-B^{-1}}_{op} =\mathcal{O}(\lambda^{-\frac{1}{2}})
\end{equation*}
where the operator norm is taken in $L^2(\Omega_2)$.
\end{theorem}

\begin{proof}
Compactness is straightforward since the set of compact operators form an algebra and both $\ran(r_{_{\Omega_2}}A^{-1}_{\lambda}e_{_{\Omega_2}})$ and $\ran(B^{-1})$ are compactly embedded in $L^2(\Omega_2)$. The rest of the proof is an adaptation of an idea from {\sc Birman \& Solomyak} \cite[pg 105]{Birman:1980uq}. Take $f,g\in L^2(\Omega_2)$ and set $u=r_{_{\Omega_2}}A^{-1}_{\lambda}e_{_{\Omega_2}}f$ and $v=B^{-1}g$. An integration by parts shows  that
\[(A_{\lambda}u,v)_{L^2(\Omega_2)} =\int\limits_{\Omega_2}A_{\lambda}uv =-\int\limits_{\Omega_2}\lap uv = \int\limits_{\Omega_2}\nabla u\cdot\nabla v - \int\limits_{\partial\Omega_2}\PD{u}{n}v\]
and
\[(u,Bv)_{L^2(\Omega_2)}  =-\int\limits_{\Omega_2}u\lap v = \int\limits_{\Omega_2}\nabla u\cdot\nabla v - \int\limits_{\partial\Omega_2}\PD{v}{n}u.\]
Using the fact that 
\[\at{v}{\Gamma_1}=\at{\PD{v}{n}}{\Gamma}=\at{\PD{u}{n}}{\Gamma} =0,\]
the integrals over $\Gamma$ vanish and we get 
\[(A_{\lambda}u,v)_{L^2(\Omega_2)} -(u,Bv)_{L^2(\Omega_2)} =\dpair{\gamma_{_0}u}{\gamma_{_1}v}_{\Gamma_1},\]
or, using the definition of $f$ and $g$, that
\[(f,B^{-1}g)_{L^2(\Omega_2)}-(r_{_{\Omega_2}}A^{-1}_{\lambda}e_{_{\Omega_2}}f,g)_{L^2(\Omega_2)} = \dpair{\gamma_{_0}u}{\gamma_{_1}v}_{\Gamma_1}.\]
Since $B$ is clearly symmetric (in fact it is self adjoint), $B^{-1}$ is symmetric. Now $e_{_{\Omega_2}}f$ has support in $\Omega_2$ and by
 Corollary \ref{cor:main_result_corr} we have that $\gamma_{_0}u=\sN\gamma_{_1}u$ which in turn implies that
\[((r_{_{\Omega_2}}A^{-1}_{\lambda}e_{_{\Omega_2}}-B^{-1})f,g) =  -\dpair{\sN\gamma_{_1}u}{\gamma_{_1}v}_{\Gamma_1}.\]
From this, it follows
\begin{align*}
|((r_{_{\Omega_2}}A^{-1}_{\lambda}e_{_{\Omega_2}}-B^{-1})f,g)|&\leq \bnorm{\mathscr{N}_{\lambda}\gamma_{_1}u}_{L^2(\Gamma_1)}
\bnorm{\gamma_{_1}v}_{L^2(\Gamma_1)}.
\end{align*}
Using the fact $v\in H^2(\Omega_2)$, the trace theorem and the operator bounds on $\sN$ we see that
\[|((r_{_{\Omega_2}}A^{-1}_{\lambda}e_{_{\Omega_2}}-B^{-1})f,g)|\leq\frac{C}{\sqrt{\lambda}}\norm{g}_{L^2(\Omega_2)}\norm{f}_{L^2(\Omega_2)}.\]
The Riesz representation theorem shows, for some positive constant $C$,
\[\norm{r_{_{\Omega_2}}A^{-1}_{\lambda}e_{_{\Omega_2}}-B^{-1}}_{op}\leq\frac{C}{\sqrt{\lambda}},\]
which completes the proof.
\end{proof}

We single out the following result which the above proof and much of what is to follow rests on:
\begin{lemma}[Green's Formula]\label{lem:green_formula}
Let $f,g\in L^2(\Omega_2)$ and set $u=r_{_{\Omega_2}}A^{-1}_{\lambda}e_{_{\Omega_2}}f$ and $v=B^{-1}g$. Then the following equivalent
formulae hold
\begin{enumerate}[$(i).$]
\item $(A_{\lambda}u,v)_{L^2(\Omega_2)} -(u,Bv)_{L^2(\Omega_2)} =\dpair{\gamma_{_0}u}{\gamma_{_1}v}_{\Gamma_1}$
\item $(f,B^{-1}g)-(r_{_{\Omega_2}}A^{-1}_{\lambda}e_{_{\Omega_2}}f,g) = \dpair{\gamma_{_0}u}{\gamma_{_1}v}_{\Gamma_1}$
\item $((r_{_{\Omega_2}}A^{-1}_{\lambda}e_{_{\Omega_2}}-B^{-1})f,g) =  -\dpair{\sN\gamma_{_1}u}{\gamma_{_1}v}_{\Gamma_1}$
\item $((r_{_{\Omega_2}}A^{-1}_{\lambda}e_{_{\Omega_2}}-B^{-1})f,g) =\dpair{\sW\gamma_{_1}B^{-1}f}{\gamma_{_1}B^{-1}g}_{\Gamma_1}$
\end{enumerate}
\end{lemma}

An easy corollary of the theorem and its proof is 
 \begin{corollary}\label{cor:bounded_spectrum}
 The spectrum $\sigma(r_{_{\Omega_2}}A^{-1}_{\lambda}e_{_{\Omega_2}}-B^{-1})$ is real and discrete and is contained in a closed interval 
 $[-\varepsilon,\varepsilon]$ where $\varepsilon=\mathcal{O}(\lambda^{-\frac{1}{2}})$.
 \end{corollary}

\begin{proof}
 Let $E_\lambda:=r_{_{\Omega_2}}A^{-1}_{\lambda}e_{_{\Omega_2}}-B^{-1}$. Theorem \ref{thm:oper_conv_diff} and its proof
 show that $E_\lambda$ is compact and symmetric and gives a bound for the operator norm. The result follows from the fact that 
the operator norm furnishes a bound for the spectral radius of a bounded operator.
\end{proof}

\subsection{Estimates for the Spectral Counting Function}
Our next application will be to refine the corollary above by obtaining estimates on the singular values/eigenvalues of $r_{_{\Omega_2}}A^{-1}_{\lambda}e_{_{\Omega_2}}-B^{-1}$.  We first recall some well known definitions which can be found for instance in \cite[Appendix 1]{Birman:1980uq}.

Let $T$ be a compact operator in a Hilbert space $H$, which we always assume to be separable. Let $s_k(T)$ be the eigenvalues of the non--negative compact operator $\sqrt{T^{*}T}$ written in non increasing order. If $T$ is positive and self adjoint, we will sometimes write $\mu_k(T)$ as the (necessarily) positive eigenvalues of $T$ also written with multiplicities and in non--increasing order. As in the Introduction we define the distribution or counting function
\begin{equation*}
N(\mu; T) =\sum_{\mu_k(T)>\mu}1, \quad \mu>0,
\end{equation*}
which counts the number of eigenvalues of $T$ which are greater than $\mu$. The rest of this subsection is devoted to deriving precise asymptotics for $N(\mu;(r_{_{\Omega_2}}A^{-1}_{\lambda}e_{_{\Omega_2}}-B^{-1}))$. We commence with the following result which is easy to deduce from Theorem \ref{thm:oper_conv_diff} and Corollary \ref{cor:bounded_spectrum} and refines the latter result:

\begin{corollary}
Let $\mu>0$ be fixed. Then there exists a $\lambda_0>0$ which depends on $\mu$ and $\Omega_2$ such that for $\lambda\geq\lambda_0$ we have $N(\mu;(r_{_{\Omega_2}}A^{-1}_{\lambda}e_{_{\Omega_2}}-B^{-1}))  =0$.
\end{corollary}

Now for $f\in L^2(\Omega_2)$, the Green's formula, i.e. Lemma \ref{lem:green_formula},  and \eqref{eqn:bound_soln} show that
\[((r_{_{\Omega_2}}A^{-1}_{\lambda}e_{_{\Omega_2}}-B^{-1})f,f)_{_{\Omega_2}} =  \dpair{\sW(\gamma_1B^{-1}f)}{ \gamma_1B^{-1}f}_{_{\Gamma_1}}.\]
We may form the Rayleigh quotient and we see, for any $f\neq 0$,
\begin{align*}
\frac{((r_{_{\Omega_2}}A^{-1}_{\lambda}e_{_{\Omega_2}}-B^{-1})f,f)_{_{\Omega_2}}}{(f,f)} &= \frac{\dpair{\sW(\gamma_1B^{-1}f)}{ \gamma_1B^{-1}f}_{_{\Gamma_1}}}{(f,f)}.
\end{align*}

Defining $S=\gamma_{_1}B^{-1}$ we see that
\begin{align*}
\frac{((r_{_{\Omega_2}}A^{-1}_{\lambda}e_{_{\Omega_2}}-B^{-1})f,f)_{_{\Omega_2}}}{(f,f)}&=\frac{\dpair{Sf}{Sf}}{(f,f)}\frac{\dpair{\sW Sf}{Sf}_{_{\Gamma_1}}}{\dpair{Sf}{Sf}}\\
&\leq\norm{S}^2_{op}\frac{\dpair{\sW Sf}{Sf}_{_{\Gamma_1}}}{\dpair{Sf}{Sf}}
\end{align*}
where, as a consequence of the trace theorem, $S$ is a bounded map from $L^2(\Omega_2)$ into $ H^{\frac{1}{2}}(\Gamma_1)$. We have already observed that $\sW$ is elliptic with a positive real principal symbol. The G$\mathring{\text{a}}$rding inequality implies that $\sW$ is  lower semi-bounded and modifying $\sW$ if necessary, we assume from now on that it is positive.
The above inequality leads to the following important result:

\begin{theorem}\label{thm:basic_spec_func_ineq}
The following inequality holds for the spectral counting function:
\begin{equation}\label{eqn:spectral_func_ineq_for_diff}
N(\mu;(r_{_{\Omega_2}}A^{-1}_{\lambda}e_{_{\Omega_2}}-B^{-1})) \leq N(\norm{S}^{-2}_{op}\mu;\mathscr{W}_\lambda).
\end{equation}
\end{theorem}
Theorem \ref{thm:basic_spec_func_ineq} is important because it allows us to reduce the study of the spectral counting function for the difference,
$r_{_{\Omega_2}}A^{-1}_{\lambda}e_{_{\Omega_2}}-B^{-1}$, to the study of the spectral function of a pseudodifferential operator on a closed compact manifold.
 The proof of Theorem \ref{thm:basic_spec_func_ineq} rests on several lemmas the first of which is
 
\begin{lemma}[See Lemma $1.15$ in \cite{Birman:1980uq}]\label{lem:birman_lemma}
Let $\mathcal{H}_i$, $i=1,2$ be two separable Hilbert spaces with norms $\norm{\cdot}_i$ and inner products $(\cdot,\cdot)_i$. Let $T_i$ be  linear compact self-adjoint maps acting in 
$\mathcal{H}_i$. Let $S:\mathcal{H}_1\rightarrow \mathcal{H}_2$ be a continuous linear map such that $(T_1u,u) =0$ for $u\in\kers(S)$. If for some real $\alpha>0$ and all $u\in \mathcal{H}_1$ such that $(T_1u,u) >0$ we have
\begin{equation*}
\frac{(T_1u,u)_1}{(u,u)_1} \leq \alpha\frac{(T_2Su,Su)_2}{(Su,Su)_2}
\end{equation*}
Then for $\mu >0$ we have that $N(\mu;T_1)\leq N(\alpha^{-1}\mu;T_2)$.
\end{lemma}
 
We need the next result which guarantees that the operator $S=\gamma_{_1}B^{-1}$ satisfies the hypothesis of Lemma \ref{lem:birman_lemma}:

\begin{lemma}\label{lem:density_of_S}
Let $S=\gamma_{_1}B^{-1}$. Then it holds that $S:L^2(\Omega_2)\to L^2(\Gamma_1)$ is compact with dense range. Moreover, $((r_{_{\Omega_2}}A^{-1}_{\lambda}e_{_{\Omega_2}}-B^{-1})f,f)=0$ for $f\in\kers(S)$, where \[\kers(S) =\{f\in L^2(\Omega_2):\exists\, u  \text{\,\,satisfying\,\,} -\lap u=f, \text{\,and\,\,}\gamma_{_0}u= \gamma_{_1}u=0 \}.\]
\end{lemma}

\begin{proof}
We establish the Lemma by showing that $S$ is a surjection onto $H^{\frac{1}{2}}(\Gamma_1)$. The result then follows because $H^{\frac{1}{2}}(\Gamma_1)$ is dense in $ L^2(\Gamma_1)$ and is compactly embedded by Rellich's theorem. As $B$ is an isomorphism, it suffices to show that for $\phi\in\mathcal{D}(\Gamma_1)$ we can find $\psi\in \dom(B)$ such that $\gamma_{_1}\psi=\phi$. Since for $f=B\psi$ this would imply that $Sf =SB\psi =\gamma_{_1}(B^{-1}B\psi) =\phi$.
Using a  partition of unity we can turn this into a local problem for a function $\phi_l\in\mathcal{D}(\Gamma_1\cap O_l)$ where $O_l$ is a coordinate patch of $\Gamma_1$ in $\Omega_2$. Identifying $\Gamma_1\cap O_l$ with $x_n=0$ in $\R^n$ we can simply choose
\[\psi_l(x',x_n) = x_n\rho(x_n)\phi_l(x')\]
where $\rho(t)\in C^{\infty}_c(\R)$ is identically one in a small neighborhood of the origin.
Finally, the Green's formula, i.e. Lemma \ref{lem:green_formula}, establishes that $((r_{_{\Omega_2}}A^{-1}_{\lambda}e_{_{\Omega_2}}-B^{-1})f,f)=0$ for $f\in\kers(S)$.
\end{proof}
\begin{remark}

We could also just appeal directly to the trace theorem here and the expert reader will recognize that the above proof 
essentially does that. 
\end{remark}
With these results at our disposal,  we turn to the
\begin{proof}[Proof of Theorem \ref{thm:basic_spec_func_ineq}]
As we have already seen, the following inequality holds
\[\frac{((r_{_{\Omega_2}}A^{-1}_{\lambda}e_{_{\Omega_2}}-B^{-1})f,f)_{_{\Omega_2}}}{(f,f)}\leq\norm{S}^2_{op}\frac{\dpair{\sW Sf}{Sf}_{_{\Gamma_1}}}{\dpair{Sf}{Sf}}.\]
With the Hilbert spaces $\mathcal{H}_1:= L^2(\Omega_2)$, $\mathcal{H}_2 :=L^2(\Gamma_1)$; the operators $T_1:=(r_{_{\Omega_2}}A^{-1}_{\lambda}e_{_{\Omega_2}}-B^{-1})$, $T_2=\mathscr{W}_\lambda$ and $S:=\gamma_{_1}B^{-1}$, a direct application of Lemma \ref{lem:birman_lemma} shows that
$N(\mu;(r_{_{\Omega_2}}A^{-1}_{\lambda}e_{_{\Omega_2}}-B^{-1})) \leq N(\norm{S}^{-2}_{op}\mu;\mathscr{W}_\lambda)$.
\end{proof}

As mentioned previously, Theorem \ref{thm:basic_spec_func_ineq} reduces the study of the spectral function of the difference of the operators in the interior to that of the operator $\sW$ on the boundary, $\Gamma_1$ and we now turn our analysis to $\sW$. The following well known result, specialized to our particular situation, is known as the Weyl asymptotic formula. The statement which follows is modified from the one in {\sc H\"ormander} \cite{Hormander:1968kx}:
 
\begin{proposition}
Let  $x'\in \Gamma_1$ and for $\mu>0$ define \[B_{x'}(\mu) =\{\xi'\in T^{*}_{x'}(\Gamma_1):\,\sigma_{-1}(\sW)>\mu\}.\] 
As $\mu\to 0$, it holds that
\[N(\mu;\sW)\sim \frac{1}{(2\pi)^{n-1}}\int\limits_{\Gamma_1}\int\limits_{B_{x'}(\mu)}d\xi'd\sigma_{x'}.\]
where $d\sigma_{x'}$ is the surface measure on $\Gamma$.
\end{proposition}

In order to apply this result, we begin by determining the principal symbol $\sigma_{-1}(\sW)$. Recall the formula \eqref{eqn:weyl_symbl_intro} 
which shows that
\begin{align*}
\sigma_{-1}(\sW)&=\frac{A_{nn}(x')}{\sqrt{A_{nn}\abs{\xi'}^2-|\nabla\chi\cdot\xi'|^2}+\sqrt{A_{nn}(\abs{\xi'}^2+\lambda)-|\nabla\chi\cdot\xi'|^2}}\\
&=\frac{\sqrt{A_{nn}(x')}}{\sqrt{\abs{\xi'}^2-|\hat{\nu}'(x')\cdot\xi'|^2}+\sqrt{(\abs{\xi'}^2+\lambda)-|\hat{\nu}'(x')\cdot\xi'|^2}}
\end{align*}
where 
\[\hat{\nu}(x') = \frac{(-\nabla_{x'}\chi,1)^t}{\sqrt{A_{nn}(x')}}\]
 is the unit normal vector to $\Gamma_1$ at $x'$ written in local coordinates and $\hat{\nu}'(x')$ denotes the first $n-1$ components. We also note that in local coordinates 
 \[d\sigma_{x'}=\sqrt{A_{nn}(x')}dx'\]
  with $dx'$ the usual Lebegue measure on $\R^{n-1}$. 

With the above notation and considerations in mind, we can state and prove 
\begin{theorem}\label{thm:spec_count_func_boundary}
We have that
 \begin{equation*}
N(\mu;\sW)\sim \frac{(4\pi)^{1-n}}{(n-1)}\int\limits_{\Gamma_1}I_{n}(x')\paren{\frac{\sqrt{A_{nn}(x')}}{\mu}-\frac{\lambda\mu}{\sqrt{A_{nn}(x')}}}^{n-1}_{+}d\sigma_{x'}
\end{equation*}
 where $q_{+} =\max\{0,q\}$ and
 \[I_{n}(x'):=\int\limits_{S^{n-2}}\frac{1}{\paren{1-|\hat{\nu}'\cdot\theta'|^2}^{\frac{n-1}{2}}}dS_{\theta'},\]
 with $\theta\in S^{n-2}$ and $dS_{\theta'}$ the measure on $S^{n-2}$.
 \end{theorem}

\begin{proof}
This proof is really a computation and it all boils down to estimates for $B_{x'}(\mu)$. By homogeneity considerations and a change of variables we have that
\[\int\limits_{B_{x'}(\mu)}d\xi'= \lambda^{\frac{n-1}{2}}\int\limits_{\Sigma_{x'}(\mu)}d\xi',\]
where
\[\Sigma_{x'}(\mu)= \left\{\xi':\sqrt{|\xi'|^2-|\hat{\nu}'\cdot\xi'|^2}+\sqrt{|\xi'|^2+1-|\hat{\nu}'\cdot\xi'|^2}<\frac{\sqrt{A_{nn}(x')}}{\sqrt{\lambda}\mu} \right\}.\]
From this, we see that $\Sigma_{x'}(\mu)$ is nonempty if and only if $1\leq\frac{\sqrt{A_{nn}(x')}}{\sqrt{\lambda}\mu}$.
In particular we note that the Lebesgue measure of $\Sigma_{x'}(\mu)$, which we denote by $|\Sigma_{x'}(\mu)|$, for fixed $\mu$ decreases as $\lambda$ increases. A direct computation now shows that
\[\Sigma_{x'}(\mu)=\left\{\xi':|\xi'|<\frac{1}{2}\paren{\frac{\sqrt{A_{nn}(x')}}{\sqrt{\lambda}\mu}-\frac{\sqrt{\lambda}\mu}{\sqrt{A_{nn}(x')}}}\frac{1}{\sqrt{1-|\hat{\nu}'\cdot\theta'|^2}}\right\}.\]
Reverting to polar coordinates we see that
\[\int\limits_{B_{x'}(\mu)}d\xi'= \lambda^{\frac{n-1}{2}}\int\limits_{S^{n-2}}\int\limits_0^{r^*}r^{n-2}\,dr\,d\theta'\]
where $r^*=\frac{\rho(x')}{\sqrt{1-|\hat{\nu}'\cdot\theta'|^2}}$ and $\rho(x')=\frac{1}{2}\paren{\frac{\sqrt{A_{nn}(x')}}{\sqrt{\lambda}\mu}-\frac{\sqrt{\lambda}\mu}{\sqrt{A_{nn}(x')}}}_{+}$. 

Carrying out the integration we obtain
\[\int\limits_{B_{x'}(\mu)}d\xi'=\frac{\lambda^{\frac{n-1}{2}}(\rho(x'))^{n-1}}{n-1}\int\limits_{S^{n-2}}\frac{1}{\paren{1-|\hat{\nu}'\cdot\theta'|^2}^{\frac{n-1}{2}}}dS_{\theta'}.\]
It is easy to verify that the last integral on the right converges. To check this, we let $t=t(x') :=|\hat{\nu}'(x')|$ and we note that $0\leq t<1$. Rotating the sphere so that 
$\hat{\nu}'(x')=t\vec{e}_1$ we see that 
\[I_{n}(x'):=\int\limits_{S^{n-2}}\frac{1}{\paren{1-|\hat{\nu}'\cdot\theta'|^2}^{\frac{n-1}{2}}}dS_{\theta'} = \int\limits_{S^{n-2}}\frac{1}{\paren{1-t^2\theta_1^2}^{\frac{n-1}{2}}}dS_{\theta'}.\]
Let $L=\text{Lip}(\Gamma_1)$ denote the Lipschitz constant for $\Gamma_1$ i.e., the supremum of $|\nabla\chi^l|$ over all the coordinate charts which locally flatten $\Gamma_1$.
Then we see that $0\leq t(x')\leq L/\sqrt{1+L^2}$ and as such $I_{n}(x')$ is uniformly bounded and also depends smoothly on $x'$. We note in passing that 
\[\omega_{n-2}\leq I_{n}(x')\leq\omega_{n-2}(A_{nn}(x'))^{\frac{n-1}{2}},\]
where $\omega_{n-2}$ is the volume of the unit sphere $S^{n-2}$. Putting it all together,  we see that 
\begin{align*}
N(\mu;\sW)&\sim \frac{1}{(2\pi)^{n-1}}\int\limits_{\Gamma_1}\int\limits_{B_{x'}(\mu)}d\xi'd\sigma_{x'}\\
&\sim \frac{(4\pi)^{1-n}}{(n-1)}\int\limits_{\Gamma_1}I_n(x')\paren{\frac{\sqrt{A_{nn}(x')}}{\mu}-\frac{\lambda\mu}{\sqrt{A_{nn}(x')}}}^{n-1}_{+}d\sigma_{x'},
\end{align*}
which proves the theorem.
\end{proof}

Finally, Theorem \ref{thm:basic_spec_func_ineq} and Theorem \ref{thm:spec_count_func_boundary} just proved give the following relatively crude
but new estimate

\begin{corollary}
For convenience set $E_\lambda:=r_{_{\Omega_2}}A^{-1}_{\lambda}e_{_{\Omega_2}}-B^{-1}$. Then, in local coordinates,
\begin{equation}\label{eqn:spec_count_func_diff}
N(\mu;E_\lambda)\sim \frac{(4\pi)^{1-n}}{(n-1)}\int\limits_{\Gamma_1}I_n(x')\paren{\frac{|\hat{\nu}_n(x')|^{-1}}{\norm{S}^{-2}_{op}\mu}-\frac{\norm{S}^{-2}_{op}\lambda\mu}{|\hat{\nu}_n(x')|^{-1}}}^{n-1}_{+}\frac{dx'}{|\hat{\nu}_n(x')|}
\end{equation}
\end{corollary}
We have not found a good physical interpretation of the above formula. Although the integrand can be given a coordinate invariant meaning (recall that the principal symbol is coordinate invariant), the expression is still unwieldy. It is still worth pointing out that \eqref{eqn:spec_count_func_diff} shows that the asymptotics depend on the geometry of the domain in the following ways:
\begin{enumerate}
\item The norm $\norm{S}_{op}$ depends on the volume of $\Omega_1$ and the distance from $\Gamma_1$ to $\Gamma$. This is because  $B^{-1}$ depends on these quantities via the (best constant in the) Poincar\'e inequality. 
\item Again let $L=\text{Lip}(\Gamma_1)$ denote the Lipschitz constant for $\Gamma_1$. It follows that $\sqrt{A_{nn}^l(x')}\leq \sqrt{1+L^2}$ and thus the integrand vanishes if $\sqrt{\lambda}\mu\geq\norm{S}^{-2}_{op}\sqrt{1+L^2}$.
\end{enumerate}

Evidently, the integral in the formula depends on the volume of $\Gamma_1$ and we end this section with the comforting observation that for large but fixed $\lambda$, we recover the ``Weylian" asymptotics as $\mu\to 0$:
\[N(\mu;(r_{_{\Omega_2}}A^{-1}_{\lambda}e_{_{\Omega_2}}-B^{-1}))\sim \frac{(4\pi)^{1-n}}{(n-1)}\omega_{n-2}\abs{\Gamma_1}\paren{\frac{\mu}{\norm{S}_{op}^{2}}}^{1-n}\sim \mathcal{O}(\mu^{1-n}).\] 

\section{Acknowledgements}
We are grateful to Andres Larrain-Hubach for reading an early draft of this article and providing useful comments which undoubtedly improved the exposition
in this work.

\thispagestyle{empty}
\bibliography{lib_papers} 

\providecommand{\bysame}{\leavevmode\hbox to3em{\hrulefill}\thinspace}
\providecommand{\MR}{\relax\ifhmode\unskip\space\fi MR }
\providecommand{\MRhref}[2]{%
  \href{http://www.ams.org/mathscinet-getitem?mr=#1}{#2}
}
\providecommand{\href}[2]{#2}
\begin{thebibliography}{10}

\bibitem{Adams:1975fk}
R.A Adams, \emph{Sobolev spaces}, Academic Press, 1975.

\bibitem{Agbanusi:2013arx}
I.C Agbanusi, \emph{Rate of convergence for large coupling limits in sobolev
  spaces}, arXiv (2014), no.~1402.3320.

\bibitem{Agbanusi:2014ys}
I.C Agbanusi and S.A. Isaacson, \emph{A comparison of bimolecular reaction
  models of stochastic reaction diffusion systems}, Bull. Math. Biol.
  \textbf{76} (2014), no.~4, 922--946.

\bibitem{Agmon:1965fk}
S.~Agmon, \emph{On kernels, eigenvalues and eigenfunctions to elliptic
  problems}, Communications on pure and applied mathematics \textbf{18} (1965),
  no.~4, 627--663.

\bibitem{Agranovich:1971fk}
M.S. Agranovich, \emph{Boundary value problems for systems with a parameter},
  Math. USSR Sbornik \textbf{13} (1971), no.~1, 25--64.

\bibitem{Agranovich:1964zr}
M.S. Agranovich and M.I. Vishik, \emph{Elliptic problems with a parameter and
  parabolic problems of general type}, Russ. Math. Surv. \textbf{19} (1964),
  no.~3, 51--157.

\bibitem{BelHadjAli:2011ys}
H.~BelHadjAli, A.~Ben Amor, and J.~F Brasche, \emph{Large coupling convergence:
  overview and new results}, Partial Differential Equations and Spectral Theory
  (M.~Demuth, ed.), Operator Theory: Advances and Applications, vol. 211,
  Birkha\"user/Springer Basel, 2011, pp.~73--117.

\bibitem{Birman:1980uq}
M.S. Birman and M.Z. Solomjak, \emph{Quantitative analysis in {S}obolev
  imbedding theorems and applications to spectral theory}, AMS Translations
  Series 2, vol. 114, AMS, 1980.

\bibitem{Bruneau:2002zr}
V.~Bruneau and G.~Carbou, \emph{Spectral asymptotic in the large coupling
  limit.}, Asymptotic Analysis \textbf{29} (2002), no.~2, 91--113.

\bibitem{Chazarain:1982fk}
J.~Chazarain and A.~Piriou, \emph{Introduction to the theory of linear partial
  differential equations}, Studies In Mathematics and Its Applications,
  vol.~14, North-Holland, 1982.

\bibitem{M.Demuth:1995zr}
M.~Demuth, W.~Kirsch, and I.~McGillivray, \emph{Schr\"odinger operators -
  geometric estimates in terms of occupation times}, Communications in Partial
  Differential Equations \textbf{20} (1995), no.~1-2, 37--57.

\bibitem{Gesztesy:1988ua}
F.~Gesztesy, D.~Gurarie, H.~Holden, M.~Klaus, L.~Sadun, B.~Simon, and P.~Vogl,
  \emph{{Trapping and cascading of eigenvalues in the large coupling limit}},
  Communications in Mathematical Physics \textbf{118} (1988), no.~4, 597--634.

\bibitem{Grubb:1968uq}
G.~Grubb, \emph{A characterization of the non local boundary value problems
  associated with an elliptic operator}, Ann. Scuola Norm. Sup. Pisa (3)
  \textbf{22} (1968), no.~3, 425--513.

\bibitem{Grubb:1978uq}
\bysame, \emph{Remainder estimates for eigenvalues and kernels of
  pseudo-differential elliptic systems.}, Math. Scand \textbf{43} (1978),
  275--307.

\bibitem{Hormander:1968kx}
L.~H\"ormander, \emph{The spectral function of an elliptic operator}, Acta
  Mathematica \textbf{121} (1968), no.~1, 193--218.

\bibitem{Kato:1980kx}
T.~Kato, \emph{Perturbation theory for linear operators}, 2nd ed.,
  Springer-Verlag, New York, 1980.

\bibitem{Bardos:1982fk}
J.~Rauch and C.~Bardos, \emph{Maximal positive boundary value problems as
  limits of singular perturbation problems}, Trans. Amer. Math. Soc.
  \textbf{270} (1982), no.~2, 377--408.

\bibitem{Robert:1998vn}
D.~Robert, \emph{Semi-classical approximation in quantum mechanics: a survey of
  old and recent mathematical results}, Helv. Phys. Acta. \textbf{71} (1998),
  44--116.

\bibitem{Seeley:1968vn}
R.T. Seeley, \emph{The resolvent of an elliptic boundary problem}, American
  Journal of Mathematics \textbf{91} (1969), no.~4, 889--920.

\bibitem{Vishik:1952kx}
M.I Vishik, \emph{On general boundary problems for elliptic differential
  equations}, Trudy Moskov. Mat. Ob\u s\u v. \textbf{1} (1952), 187--246, Amer.
  Math. Soc. Trans. (2) 24, 1968, 107-172.

\bibitem{Vishik:1960uo}
M.I. Vishik and L.A. Lyusternik, \emph{{The asymptotic behaviour of solutions
  of linear differential equations with large or quickly changing coefficients
  and boundary conditions}}, Russian Mathematical Surveys (1960).

\end{thebibliography}
\bibliographystyle{amsplain}
\end{document}